\documentclass[a4paper,psamsfonts,reqno]{amsart}
\usepackage{amssymb,amscd,amsmath}

\usepackage[latin5]{inputenc}
\newtheorem{theorem}{Theorem}[section]
\newtheorem{lem}[theorem]{Lemma}
\newtheorem{thm}[theorem]{Theorem}
\newtheorem{rem}[theorem]{Remark}
\newtheorem{prop}[theorem]{Proposition}
\newtheorem{cor}[theorem]{Corollary}
\newtheorem{defn}[theorem]{Definition}

\newtheorem{exs}[theorem]{Examples}

\def\z*{\mbox{Z$^*$}}
\def\s*{\mbox{(S$^*$)}}

\def\X{\mathcal{X}}
\def\Y{\mathcal{Y}}
\def\Z{\mathcal{Z}}

\begin{document}

\title[Examples of generalized perfect rings]{A family of examples of \\ generalized perfect rings}

\author{P\i nar Aydo\u gdu}
\address{
Department of Mathematics\\ Hacettepe University\\ 06800 Beytepe, Ankara\\ Turkey}
\email{paydogdu@hacettepe.edu.tr}

\author{Dolors Herbera}
\address{Departament de Matem\`atiques \\
Universitat Aut\`onoma de Barcelona \\ 08193 Bellaterra (Barcelona),
Spain} \email{dolors@mat.uab.cat}

\thanks {The second named author   acknowledges partial support
from   DGI MINECO
MTM2011-28992-C02-01, by ERDF UNAB10-4E-378
\emph{A way to build Europe}, and by the
Comissionat per Universitats i
Recerca de la Generalitat de Catalunya Project
2009 SGR 1389. Part of this paper was written while she was visiting Hacettepe University supported by a grant of the The Scientific and Technological Research Council of Turkey: T\"ubitak; she thanks her host for the kind hospitality. \protect\newline 2010 Mathematics Subject
Classification. Primary: 16D40  Secondary: 16E99.}

\maketitle
\begin{abstract}
 We construct a family of semiprimitive and non von Neumann regular   rings satisfying that any right or left module is isomorphic to a quotient of its    flat cover (in the sense of Enochs) by a small submodule. This answers in the negative a question posed by A.~Amini, B.~Amini, M.~Ershad and H.~Sharif (2007).
 \end{abstract}

\section{Introduction}

Inspired by the fundamental work of Bass \cite{bass} on perfect rings and projective covers, A.~Amini, B.~Amini, M.~Ershad and H.~Sharif proposed in \cite{A} to study a class of rings that they named \emph{generalized perfect rings}.

Let $R$ be a ring, and let $F$ and $M$ be right $R$-modules such that $F_R$ is flat. Following \cite{A}, a module epimorphism  $f\colon F\rightarrow M$   is said to be   a \emph{$G$-flat cover} of $M$  if $\mathrm{Ker}\, (f)$ is a small submodule of $F$. Still following \cite{A}, a ring $R$ is called {\em right generalized perfect} (right $G$-perfect, for short) if every right $R$-module has a flat cover. A ring $R$ is called {\em $G$-perfect}
if it is both left and right $G$-perfect.

It is clear from the definition that right perfect rings are right $G$-perfect rings, and also that von Neumann regular rings are $G$-perfect rings. Moreover, the class of right $G$-perfect rings is closed under finite products and quotients \cite[Proposition~2.6]{A}.

A celebrated result by Bican, El Bashir and Enochs  \cite{BEE} shows  that any module has a  flat cover (see Definition~\ref{approximation}). The relation between $G$-flat covers and flat covers (if any!) is quite unclear. In the case of perfect rings they coincide, and in the case of von Neumann regular rings flat covers are trivially $G$-flat covers but, in general, the converse is not true  \cite{A} (it happens that flat covers are unique up to isomorphism, while $G$-flat covers are not!).

Looking for a characterization of $G$-perfect rings,  it was shown in \cite{A} that if $R$ is right $G$-perfect, then the Jacobson radical $J(R)$ is right $T$-nilpotent and, hence,  idempotents lift modulo   $J(R)$. Moreover, it was  also proved that if $R$ is right duo (i.e. all right ideals are two-sided ideals) and right $G$-perfect, then $R/J(R)$ is von Neumann regular. It was claimed that it was reasonable to conjecture that a right $G$-perfect ring is   von Neumann regular modulo the Jacobson radical. In this work, we answer this conjecture in the negative by constructing semiprimitive $G$-perfect rings that are not von Neumann regular (see Examples~\ref{ex1}).

Our examples are built using the following well-known pattern: let $S\hookrightarrow T$ be a ring inclusion,   and consider the ring
\[R=\{ (x_1,x_2,\ldots, x_n,x,x,\ldots)| n\in \mathbb{N}, x_i\in T, x\in S\}\subseteq T^{\mathbb{N}}.\]
Such a construction appears quite frequently in the literature. Our new input is the study of its   category of modules. To do that it is very useful to consider a family of $\mathrm{TTF}$-triples associated to such type of rings that relates the category of modules over $R$ with the categories of $T$-modules and of $S$-modules, cf. Remark~\ref{propertiesT}.
In Theorem~\ref{main}, we show that if $S$ is right $G$-perfect  and $T$ is von Neumann regular, then $R$ is also right $G$-perfect. We also show that if  flat covers of $S$-modules are $G$-flat covers, then the same is true for $R$. Therefore, as far as we know, it could well be that for a $G$-perfect ring  flat covers are always $G$-flat covers.

We have structured the paper in two sections. In the first one, we give some background material and we prove some preliminary results. In the second section, we specialize to the  particular class of examples we are interested in. For the  reader's sake, we have made the paper as self-contained as possible.

We are very grateful to the referee for his/her comments. Specially, for pointing out that the use of $TTF$-triples would give more  light to our first and naive approach.

\section{Background and preliminary results}

All our rings are associative with $1$, and ring homomorphisms preserve $1$. The term module means unital module.

We recall the notion of cover. The definition is  due to Auslander for the case of artin algebras and  to Enochs for arbitrary rings.

\begin{defn}\label{approximation} {\rm Let $R$ be a ring. Let $\mathcal{C}$ be a class of right $R$-modules, and let $M_R$ be a right $R$-module. A  module homomorphism $f\colon C\to M$ is a \emph{$\mathcal{C}$-precover} of $M$ if it satisfies that
\begin{enumerate}
\item[(i)] $C\in \mathcal{C}$;
\item[(ii)] $\mathrm{Hom}_R(C',f)$ is onto for any $C'\in \mathcal{C}$.
\end{enumerate}
The homomorphism $f$ is a \emph{$\mathcal{C}$-cover} if, in addition, it is right minimal.

Recall that $f\colon C\to M$ is said to be \emph{right minimal} if for any $g\in \mathrm{End}_R(C)$, $f=fg$ implies  $g$   bijective (cf. \cite[\S 1.2]{ARS}).}
\end{defn}

We are interested in the case $\mathcal{C}=\mathcal{F}$ is the class of flat right modules, that is on   flat precovers and on  flat covers. We also consider the class $$\mathcal{E}=\{ B\in \mathrm{Mod}\mbox{-}R\mid \mathrm{Ext}_R^1(L,B)=0\mbox{ for any flat right $R$-module $L$}\}$$ which is called the class of \emph{(Enochs) cotorsion modules}.

Let us recall some well known facts on flat covers that we will use in the sequel. Since any module is a homomorphic image of a projective module it easily follows that  any flat precover is onto. As the class of flat modules is closed under extensions,  Wakamatsu Lemma \cite[Lemma~2.1.13]{GT} implies that the kernel of any flat cover is a cotorsion module. On the other hand, a result of Eklof and Trlifaj \cite[Theorem~10]{ET}   implies that any right $R$-module $M$ fits into an exact sequence
\[0\to B\to L\stackrel{g}\to M\to 0\]
where $L$ is flat and $B$ is a cotorsion module, cf. \cite{BEE}. Since $\mathrm{Ext}_R^1(L',B)=0$ for any flat module $L'$,  it follows that  $g$ is a   flat precover.

If $R$ is a right perfect ring then flat right $R$-modules are projective, and hence  flat covers coincide with projective covers and with $G$-flat covers.   In general, covers are unique up to isomorphism while $G$-flat covers are not \cite[Example~3.1]{A}.

The following definition is quite well known.

\begin{defn} Let $R$ be a ring. A pair $(\X , \Y)$ of subclasses  of $\mathrm{Mod}$-$R$ is said to be a torsion pair if
\begin{enumerate}
\item[(i)] $\mathrm{Hom}_R(X,Y)=\{0\}$ for any $X\in \X$ and $Y\in \Y$.
\item[(ii)] If $X_R$ is a right $R$-module such that $\mathrm{Hom}_R(X,Y)=\{0\}$ for any   $Y\in \Y$ then $X\in \X$.
\item[(iii)] If $Y_R$ is a right $R$-module such that $\mathrm{Hom}_R(X,Y)=\{0\}$ for any   $X\in \X$ then $Y\in \Y$.
\end{enumerate}
In this case, $\X$ is said to be a \emph{torsion class} and $\Y$ is a \emph{torsion-free class}. The objects of $\X$ are called torsion modules and the objects in $\Y$ are called torsion-free modules.
\end{defn}

We refer to \cite{S} for the basics on  torsion pairs and their uses. We quickly recall the facts on torsion pairs that are more relevant to us.

Let  $(\X , \Y)$ be a torsion pair. If $M_R$ is a right $R$-module, it makes sense to consider the largest submodule of $M_R$ that is an object of $\X$, such submodule is called the torsion submodule of $M$ and is denoted by   $t(M)$. It is not difficult to see that $t$ is indeed a functor and a radical. So that, there is an exact sequece
\[0\to t(M)\to M\to M/t(M)\to 0\]
where $M/t(M)\in \Y$.

A class of modules $\X$ is torsion if and only if it is closed under isomorphisms, extensions, coproducts and quotients. Dually, a class of modules $\Y$ is a torsion-free class if it is closed under isomorphism, extensions, submodules and products.

Notice that if a class of modules $\Y$ is closed by products, coproducts, subobjects, quotients and extensions then $\Y$ is a torsion class and a torsion free class at the same time. Therefore, one has a triple $(\X , \Y ,\Z)$ such that $(\X , \Y)$ and $(\Y ,\Z)$ are torsion pairs. Such a triple is called a $\mathrm{TTF}$-triple \cite[Ch. VI \S 8]{S}.

\bigskip

Let
$0\longrightarrow M\overset{h}{\longrightarrow }N\overset{f}{\longrightarrow
}K\longrightarrow 0$
be an exact sequence of right $R$-modules and let $L\overset{g}{\longrightarrow }K\longrightarrow 0$ be an onto homomorphism. We consider the pullback of $f$ and $g$ to obtain a commutative diagram with exact rows and columns:

 \begin{equation}\label{pullback}
\begin{array}{cll}
& 0 & 0 \\
& \downarrow  & \downarrow  \\
& X \phantom{M}= & X=Kerg \\
& \downarrow ^{\varepsilon _{2}} & \downarrow  \\
0\longrightarrow M\overset{\varepsilon _{1}}{\longrightarrow } & L^{\prime }%
\overset{\pi _{2}}{\longrightarrow } & L\longrightarrow 0 \\ \phantom{l} \shortparallel
& \downarrow ^{\pi _{1}} & \downarrow ^{g} \\
0\longrightarrow M\underset{h}{\longrightarrow } & N\underset{f}{%
\longrightarrow } & K\longrightarrow 0 \\
& \downarrow  & \downarrow  \\
& 0 & 0%
\end{array}%
\end{equation}

\bigskip
where $L'=\{(x,y)\in N\oplus L|f(x)=g(y)\}$. The maps $\pi_1\colon L'\rightarrow N$ and $\pi_2\colon L'\rightarrow L$ are restrictions of the canonical projections
$\pi_1\colon N\oplus L\rightarrow N$ and $\pi_2\colon N\oplus L\rightarrow L$, respectively. The homomorphism $\varepsilon_1\colon M\rightarrow L'$ is defined by $\varepsilon_1(x)=(h(x),0)$ for each $x\in M$, and $\varepsilon_2\colon X\rightarrow L'$ is defined by $\varepsilon_2(y)=(0,y)$ for each $y\in X$.

Our construction of $G$-flat covers will be done with such a pullback diagram. For further use we note the following lemma.

\begin{lem}\label{smallness} With the notation above, assume that  $X$ is a small submodule of $L$. If $Y\le L'$ is such that $\varepsilon_1(M)\subseteq Y$, then  $\varepsilon_2(X)+Y=L'$ implies $L'=Y$.
\end{lem}
\begin{proof} Since $X\ll L$,  $\pi_2(Y)=L$, because $\pi_2(L')=\pi_2(\varepsilon_2(X))+\pi_2(Y)=X+\pi_2(Y)=L$.

Let $(x,y)\in L'$. Since $\pi_2(Y)=L$, there exists $(x',y)\in Y$  such that $g(y)=f(x')=f(x)$. Therefore, $x-x'\in Im(h)$, and so
$(x-x',0)\in \varepsilon_1(M)\subseteq Y$. Now, $(x,y)=(x',y)+(x-x',0)\in Y$. This proves that $Y=L'$.
\end{proof}

The next two lemmas  will be useful to produce $G$-flat covers  and flat covers by using  diagram~(\ref{pullback}).

\begin{lem} \label{smalltorsion} Let $R$ be a ring. Let $(\mathcal{X},\mathcal{Y})$ be a  torsion pair in $\mathrm{Mod}$-$R$ such that the associated torsion radical $t$ is exact. Assume that in diagram~(\ref{pullback}), $M\in \X$ and   $K$, $L\in \Y$.

If $X$ is small in $L$ then $\varepsilon _2 (X)$ is small in $L'$. In particular, if $L_R$ and $M_R$ are flat, then $\pi _1\colon L'\to N$ is a $G$-flat cover of $N$.
\end{lem}

\begin{proof} Since $L\in \Y$ which is a torsion-free class, also $X\in \Y$. Let $Y\le L'$ is such that $\varepsilon _2(X)+Y=L'$.  Applying $t$ to the exact sequence
\[0\to \varepsilon _2(X)\cap Y \to  \varepsilon _2(X)\oplus  Y\to L'\to 0\]
yields that $t(L')= t(\varepsilon _2(X))\oplus t(Y)=t(Y)$. Since $M\in \X$ and $L\in \Y$ it follows that $t(L')=\varepsilon _1(M)$. Therefore, $\varepsilon _1(M)\subseteq Y$. Now it follows from Lemma~\ref{smallness} that $\varepsilon _2(X)$ is small in $L'$.

The rest of the claim is clear.
\end{proof}

\begin{lem} \label{minimaltorsion} Let $R$ be a ring, let $(\mathcal{X},\mathcal{Y})$ be a  torsion pair in $\mathrm{Mod}$-$R$. Assume that in diagram~(\ref{pullback}), $M\in \X$ and   $K$, $L\in \Y$.  Then $g$ is right minimal if and only if $\pi _1$ is right minimal.
\end{lem}

\begin{proof} Throughout the proof we denote by $t$ the torsion radical associated to the torsion pair. Since $M\in \X$ and $\L \in \Y$, $t(L')\cong M$ and $L'/t(L')\cong L$. Also, because of diagram~(\ref{pullback}), $t(\pi _1)\colon t(L')\to t(N)$ is an isomorphism. Fix an isomorphism $\varphi \colon L\to L'/t(L')$

Assume $g$ is right minimal, and let $h'\colon L'\to L'$ be such that $\pi _1 h'=\pi _1$. Then $t(\pi _1)t(h')=t(\pi _1)$, and since  $t(\pi _1)$ is an isomorphism we deduce that $t(h')$  is an isomorphism. Now to conclude that $h'$ is an isomorphism we only need to prove that the induced map $h\colon L'/t(L')\to L'/t(L')$ is an isomorphism. Indeed,
\[ g\pi _2=f\pi _1 =g\pi _2 h'.\]
Since $\pi _2 h'=h\varphi \pi _2$ and $\pi _2$ is an onto map, $g=gh\varphi$. Therefore $h\varphi$ and, hence, $h$ are isomorphisms.

The converse follows using similar arguments. \end{proof}

The following lemma is a particular case of \cite[Proposition 4.1.3]{CE}.

\begin{lem} \label{flatoverring} Let $\varphi\colon R\to S$ be a ring morphism such that ${}_RS$ is flat. Let $X_S$ and $Y_S$ be right $S-modules$, then $\mathrm{Ext}_R^i(X,Y)\cong \mathrm{Ext}_S^i(X,Y)$ for any $i\ge 0$.
\end{lem}


The following proposition collects some well known facts on $\mathrm{TTF}$-triples that will be useful in the sequel.

\begin{prop}\label{recognizing}  Let $R$ and $S$ be rings such that there is an exact sequence
\[0\to I\to R\stackrel{\varphi}\to S\to 0\]
where $\varphi$ is a ring morphism such that ${}_RS$ becomes a flat module. Consider the following classes of modules
\[\mathcal{X}=\{X\in \mathrm{Mod}\mbox{-}R\mid XI=X\} \]
\[\mathcal{Y}=\{Y\in \mathrm{Mod}\mbox{-}R\mid YI=\{0\}\} \]
\[\mathcal{Z}=\{Z\in \mathrm{Mod}\mbox{-}R\mid \mathrm{ann}_Z(I)=\{0\} \} \]
then $(\mathcal{X},\mathcal{Y},\mathcal{Z})$ is a $\mathrm{TTF}$-triple such that the torsion pair $(\mathcal{X}, \mathcal{Y})$ is hereditary and   $\mathrm{Ext}_R^i(X,Y)=0$ for any $i\ge 0$, $X\in \mathcal{X}$ and $Y\in \mathcal{Y}$.

Moreover, the torsion radical associated to the torsion class $\mathcal{X}$ is naturally equivalent to the exact functor $-\otimes _R I$, and the torsion radical associated to the class $\mathcal{Y}$ is naturally equivalent to the functor $\mathrm{Hom}_R(S,-)$.
\end{prop}

\begin{proof} Notice that the class $\Y$ is equivalent to the category of right $S$-modules, so that $\Y$ has all the closure properties that ensure that it is a torsion class and a torsion-free class. Then it is easy to check that the classes $\X$ and $\Z$ are the ones that complete the $\mathrm{TTF}$-triple.

Since ${}_RS$ is flat, ${}_RI$ is a pure submodule of $R$ and, hence,  it is also flat and an idempotent ideal. Therefore, if $t$ denotes the radical associated to the torsion class $\X$, it follows that $t(M)=MI\cong M\otimes _RI$ for any right $R$-module $M$. From this it easily follows that $t$ is naturally equivalent to the exact funcor $-\otimes _RI$.

Since $t$ is exact, $\X$ is closed under submodule and then $\Y$ is closed by injective hulls \cite[Proposition~VI.3.1]{S}. Therefore, if $Y\in \Y$, then its injective hull as an $R$-module must be an $S$-module, so its injective hull as an $R$-module must coincide with its injective hull as an $S$-module. This implies that, for $Y\in \Y$, we can consider an injective $R$-coresolution $I ^\bullet$ of $Y$ consisting of injective modules in $\Y$. Now, if $X\in \X$, then the cocomplex $\mathrm{Hom}_R(X,I^\bullet)$ is zero, so that it has zero cohomology. This shows that $\mathrm{Ext}_R^i(X,Y)=0$ for any $i\ge 0$.
\end{proof}

\begin{cor} \label{flatcotorsion} Let $R$ and $S$ be rings such that there is an exact sequence
\[0\to I\to R\stackrel{\varphi}\to S\to 0\]
where $\varphi$ is a ring morphism such that $S$ becomes a flat $R$-module on the right and on the left. Then:
\begin{itemize}
\item[(i)] $M_R$ is flat if and only if $M\otimes _RS$ is a flat right $S$-module and $MI$ is a flat right $R$-module.
\item[(ii)] Let $M$ be a right $S$-module, then $M$ is cotorsion as a right $R$-module if and only if it is cotorsion as an $S$-module. \end{itemize}
\end{cor}

\begin{proof}  Notice that since $S$ is flat on both sides as $R$-module then also $I$ is a flat $R$-module on both sides.

$(i)$. By Proposition~\ref{recognizing} and following the notation in that proposition,  for any right $R$-module $M$ there is an exact sequence
\[0\to MI\to M\to M/MI\to 0\qquad (*)\]
with $MI \in \mathcal{X}$ and $M/MI \in \mathcal{Y}$.

Assume $M_R$ is flat. The functors $MI\otimes _R-$ and $M/MI\otimes _R-$ are naturally equivalent to $M\otimes _RI\otimes _R-$ and $M\otimes _RS\otimes _R-$, respectively. These functors are exact because $I_R$ and $S_R$ are flat. Therefore, $MI$ and $M/MI$ are flat right $R$-modules. In particular, $M/MI\cong M\otimes _RS$ is flat as a right $S$-module.

Assume now that $MI$ is a flat right $R$-module  and that $M/MI$ is flat as a right $S$-module. The functor  $M/MI\otimes _R- $ is exact because it is naturally equivalent to $M/MI\otimes _SS\otimes _ R-$ and $S_R$ is flat. Therefore, $M/MI$ is flat as a right $R$-module.  Using the exact sequence $(*)$ we deduce that $M_R$ is flat.

$(ii)$. By $(i)$ and Proposition~\ref{recognizing}, a right $R$-module $M$ is cotorsion if and only if $\mathrm{Ext}_R^i(F,M)=0$ for any flat module $F$ that is either in $\mathcal{X}$ or in $\mathcal{Y}$. If, in addition, $M$ is a right $S$ module, then it is in $\mathcal{Y}$ so we only  need to check that for any flat $R$-module $F$ in $\mathcal{Y}$, that is, for any flat right $S$-module, $\mathrm{Ext}_R^1(F,M)= \mathrm{Ext}_S^1(F,M)=0$ (cf. Lemma~\ref{flatoverring}). This yields that a right $S$-module is cotorsion as an $R$-module if and only if it is cotorsion as an $S$-module.
\end{proof}

\section{The examples}

Now we start to study the particular construction we are interested in.

\begin{prop} \label{TTFproperties} Let $S\subseteq T$ be an extension of rings. Let
$$R=\{ (x_1,x_2,\ldots, x_n,x,x,\ldots)| n\in \mathbb{N}, x_i\in T, x\in S\}.$$ Then, the following statements hold.
\begin{itemize}
\item[(i)] The map $\varphi \colon R \to S$ defined by $\varphi(x_1,x_2,\ldots, x_n,x,x,\ldots)=x$ is a ring homomorphism with kernel $$I=  \bigoplus_{\mathbb{N}} T=\bigoplus _{i\in \mathbb{N}} e_i R,$$ where  $e_i=(0,\ldots,0,1^{(i)},0,0,\ldots)$ for any $i\in \mathbb{N}$.
\item[(ii)]  $I$ is a two-sided, countably generated idempotent   ideal of R which is pure and projective on both sides. Therefore, $S$ is flat as a right and as a left $R$-module.
\item[(iii)] For any $i\in \mathbb{N}$, the canonical projection into the $i$-th component $\pi _i \colon R\to T$ has kernel $(1-e_i) R$ so that $T$ is projective as a right and as a left $R$-module via the $R$-module structure induced by $\pi _i$.
\end{itemize}
\end{prop}

\begin{proof} Statement $(i)$ is clear. To prove $(ii)$ note that since $\{e_i\}_{i\in  \mathbb{N}}$ is a family of orthogonal central idempotents of $R$, $I^2=I$,  $I$ is projective on both sides and, moreover, $I$ is a  pure ideal of $R$ on both sides. Now the existence of the  pure exact sequence $$0\longrightarrow I \longrightarrow R \overset{\varphi}{\longrightarrow}S\longrightarrow 0$$
yields that $S$ is flat as a right and as a left $R$-module.

Statement $(iii)$ is clear.
\end{proof}

\begin{rem} \label{propertiesT} Let $R$ be a ring as in Proposition~\ref{TTFproperties}.  In view of Proposition~\ref{recognizing}, there is a $\mathrm{TTF}$-triple  $(\mathcal{X},\mathcal{Y},\mathcal{Z})$ associated to the pure exact sequence
\[0\to I\to R\stackrel{\varphi}{\to}S\to 0\]
where
$\mathcal{X}=\{X\in \mathrm{Mod}-R\mid X=\oplus _{i\in \mathbb{N}}Xe_i\}$,
$\mathcal{Y}=\{Y\in \mathrm{Mod}-R\mid YI=\{ 0\} \} $
$\mathcal{Z}=\{Z\in \mathrm{Mod}-R\mid \mathrm{ann}_Z(I)=\{0\} \} $.
Also, for any $i\in \mathbb{N}$, the split sequence
\[0\to R(1-e_i)\to R\stackrel{\pi _i}{\to}T\to 0\]
yields a corresponding (split) $\mathrm{TTF}$-triple  $(\mathcal{X}_i,\mathcal{Y}_i,\mathcal{Z}_i)$.
 \end{rem}

\begin{prop} \label{ringproperties} Let $R$ be a ring as in the statement of Proposition~\ref{TTFproperties}. Then,
\begin{itemize}
\item[(i)] $J(R)$ contains $J=\bigoplus_{\mathbb{N}} J(T)$. Moreover, $J$ is essential on both sides into $J(R)$. In particular, $J(R)=0$   if and only if $J(T)=0$.
\item[(ii)]  $R$ is von Neumann regular if and only if $S$ and $T$ are von Neumann regular.
\end{itemize}
\end{prop}

\begin{proof}
$(i).$ Let $x\in J(T)$. Then, for any $i\in \mathbb{N}$, $x(i)=(0,\ldots,0,x^{(i)},0,0,\ldots)\in J$ is an element of $J(R)$. Since $J(R)$ is an ideal, it follows that it contains $J$.  Now we show that $J_R\le _e J(R)_R$, the statement on the left follows by a symmetric argument.

Let $0\neq a= (x_1,x_2,\ldots, x_n,x,x,\ldots)\in J(R)$. There exists $i\in \mathbb{N}$ such that $0\neq ae_i=(0,\ldots,0,y^{(i)},0,0,\ldots)\in J(R)$ for some $y\in T$. Then, for any $t\in T$, $(1,\dots ,1, \dots )-ae_i(0,\ldots,0,t^{(i)},0,0,\ldots)$ is an invertible element of $R$ and this is equivalent to say that $1-yt$ is an invertible element of $T$. Hence $y\in J(T)$, so that $ae_i$ is a non-zero element of $J\cap aR$.

$(ii).$ Assume that $R$ is von Neumann regular. Since $S$ is a homomorphic image of $R$, it is also von Neumann regular.

The canonical projection $\pi _1\colon R\to T$ defined by $\pi_1(x_1,x_2,\ldots, x_n,x,x,\ldots)=x_1$ is also an onto ring homomorphism, hence $T$ is von Neumann regular.

It is clear that if $S$ and $T$ are von Neumann regular, then so is $R$.\end{proof}

\begin{lem} \label{flat} Let $R$ be a ring as in the statement of Proposition~\ref{TTFproperties}. Let $M_R$ be a right $R$-module. Then $M_R$ is flat if and only if $M\otimes _RS$ is a flat right $S$-module and, for any $i\in \mathbb{N}$, $Me_i$ is a flat right $T$-module.
\end{lem}

\begin{proof} Applying  Corollary~\ref{flatcotorsion} with the flat epimorphism $\varphi$, it follows that $M_R$ is flat if and only if $M\otimes _RS$ is a flat right $S$-module and $ Me_i$ is a flat right $R$-module for any $i\in \mathbb{N}$.

 Applying  Corollary~\ref{flatcotorsion} with the flat epimorphism $\pi _i$, it follows that $Me_i$ is a flat right $R$-module if and only if it is a flat $T$-module. \end{proof}

Now we are ready to prove our main result:

\begin{thm} \label{main} Let $S\subseteq T$ be an extension of rings. Assume $T$ is von Neumann regular and that $S$ is right $G$-perfect. Then
$$R=\{ (x_1,x_2,\ldots, x_n,x,x,\ldots)| n\in \mathbb{N}, x_i\in T, x\in S\}$$ is a right $G$-perfect ring such that $J(R)=0$.

Moreover, if S is a ring such that flat covers are $G$-flat covers, then also  $R$ satisfies this property.
\end{thm}

\begin{proof}   By Proposition~\ref{ringproperties}, it readily follows that $J(R)=0$.

Let $N$ be any right $R$-module. By Proposition~\ref{TTFproperties}, there is a pure exact sequence
$$0\longrightarrow NI\cong \bigoplus _{i\in \mathbb{N}} Ne_i\longrightarrow  N\stackrel{f}{\longrightarrow} N/NI\longrightarrow 0.$$

 Since $T$ is von Neumann regular, for any $i\in \mathbb{N}$, $Ne_i$ is a flat $T$-module. Hence $NI$ is flat as a right $R$-module by Lemma~\ref{flat}.

 Let $0\to X\to L\stackrel{g}\to N/NI\to 0$ be a $G$-flat cover of the right $S$-module $N/NI$.
Considering the pullback of $h$ and $f $ yields the following diagram with exact rows and columns

\bigskip

$$%
\begin{array}{clcll}
& 0 &  & 0 &  \\
& \downarrow  &  & \downarrow  &  \\
& X & = & X=Kerh &  \\
& \downarrow  &  & \downarrow  &  \\
0\longrightarrow NI\longrightarrow  & L^{\prime } & \overset{%
\pi _{2}}{\longrightarrow } & L\text{\quad }\longrightarrow 0 &  \\
\shortparallel  & \downarrow ^{\pi _{1}} &  & \downarrow ^{g} &  \\
0\longrightarrow NI\longrightarrow  & N & \overset{f}{\longrightarrow } & N/NI\longrightarrow 0 &  \\
& \downarrow  &  & \downarrow  &  \\
& 0 &  & 0 &
\end{array}%
$$
Since, by Proposition~\ref{TTFproperties}, the radical associated to the torsion pair $(\X, \Y)$ is exact and $L\in \Y$, it follows from Lemma~\ref{smalltorsion} that  $\pi _1$ is a $G$-flat cover of $N$.

Now assume,  in addition, that  $0\to X\to L\stackrel{h}\to N/NI\to 0$ is a  flat cover of  the right $S$-module $N\otimes _RS$. In particular, $X_S$ is cotorsion. By Corollary~\ref{flatcotorsion}, $X_R$ is also a cotorsion module, hence $0\to X\to L'\stackrel{\pi _1}\to N\to 0$ is a   flat precover of $N$. By Lemma~\ref{minimaltorsion} it follows that $\pi _1$   is also a  flat cover.
\end{proof}

Finally, we show how to produce the examples claimed in the Introduction from Theorem~\ref{main}.

\begin{exs}\label{ex1}

\emph{\textbf{(a)}} Let $F$ be a field, and let  $S$ be any finite dimensional $F$-algebra such that $J(S)\neq 0$.  Since $S$ is artinian, it is $G$-perfect. If $\mathrm{dim}_F(S)=n$, then $S\hookrightarrow T=\mathbb{M}_n(F)$ which is von Neumann regular.
Therefore, $$R=\{ (x_1,x_2,\ldots, x_n,x,x,\ldots)| n\in \mathbb{N}, x_i\in T, x\in S\}$$ is $G$-perfect by Theorem~\ref{main}. By Proposition~\ref{ringproperties}, $J(R)=0$ and $R$ is not von Neumann regular.

For a particular realization of such a ring $R$ consider, for example, $S=\left(
\begin{array}{cc}
F & F \\
0 & F%
\end{array}%
\right) $. In this case, $T$ can be taken   to be $M_2(F)$.

\medskip

\emph{\textbf{(b)}}  Let $R$ be as in (a). Then, $R\subseteq \prod \mathbb{M}_n(F)=T'$ which is a von Neumann regular ring and
$R'=\{ (x_1,x_2,\ldots, x_n,x,x,\ldots)| n\in \mathbb{N}, x_i\in T', x\in R\}$ is also a $G$-perfect ring.

\end{exs}

Note that, by Theorem~\ref{main}, in the above  examples    flat covers are $G$-flat covers.

\end{document}